\documentclass[12pt]{amsart}
\usepackage{amsfonts,amssymb,latexsym,amsmath,verbatim,url,color,mathrsfs,tikz,extarrows}
\usepackage[margin=1in]{geometry}
\usepackage[pdfpagelabels,bookmarksopen=true]{hyperref}
\usepackage[nameinlink]{cleveref}
\usetikzlibrary{matrix,arrows}
    \DeclareFontFamily{U}{wncy}{}
    \DeclareFontShape{U}{wncy}{m}{n}{<->wncyr10}{}
    \DeclareSymbolFont{mcy}{U}{wncy}{m}{n}
    \DeclareMathSymbol{\Sh}{\mathord}{mcy}{"58} 
\DeclareMathOperator{\Cl}{Cl}
\DeclareMathOperator{\id}{id}
\DeclareMathOperator{\Br}{Br}
\DeclareMathOperator{\GL}{GL}
\DeclareMathOperator{\PGL}{PGL}
\DeclareMathOperator{\Aut}{Aut}
\DeclareMathOperator{\End}{End}
\DeclareMathOperator{\Hom}{Hom}
\DeclareMathOperator{\Pic}{Pic}
\DeclareMathOperator{\rank}{rank}
\DeclareMathOperator{\Spec}{Spec}
\DeclareMathOperator{\coker}{coker}

\DeclareMathOperator{\fppf}{fppf}
\DeclareMathOperator{\im}{im}
\def\G{\mathbf{G}}
\def\Z{\mathbf{Z}}
\def\N{\mathbf{N}}
\def\H{\mathrm{H}}
\def\A{\mathbf{A}}
\def\P{\mathbf{P}}
\def\Q{\mathbf{Q}}
\def\F{\mathbf{F}}
\usepackage[cal=boondoxo, scr=boondoxo]{mathalfa}

\theoremstyle{theorem}
\numberwithin{equation}{subsection}
\newtheorem{theorem}[subsection]{Theorem}
\newtheorem{corollary}[subsection]{Corollary}
\newtheorem{proposition}[subsection]{Proposition}
\newtheorem{lemma}[subsection]{Lemma}

\newtheorem{setup}[subsection]{Setup}

\theoremstyle{definition}
\newtheorem{definition}[subsection]{Definition}
\newtheorem{remark}[subsection]{Remark}
\newtheorem{question}[subsection]{Question}
\newtheorem{example}[subsection]{Example}

\newtheoremstyle{pgstyle} {} {} {} {} {} {.} { } {\textbf{\thmname{#1}\thmnumber{#2}}\thmnote{ (#3)}}
\theoremstyle{pgstyle}
\newtheorem{pg}[subsection]{}

\newcommand{\mf}[1]{\mathfrak{#1}}
\newcommand{\ms}[1]{\mathscr{#1}}
\newcommand{\mc}[1]{\mathcal{#1}}
\newcommand{\mb}[1]{\mathbf{#1}}
\newcommand{\mr}[1]{\mathrm{#1}}
\newcommand{\ml}[1]{\mathsf{#1}}

\begin{document}
\title[$\Br = \Br'$ question for some classifying stacks]{The $\Br = \Br'$ question for some classifying stacks}
\author{Minseon Shin}
\begin{abstract} In this paper we consider the $\Br = \Br'$ question for classifying stacks by various group schemes. These are algebraic stacks that do not necessarily admit a finite flat cover by a scheme for which $\Br = \Br'$ holds, hence are not amenable to the usual argument of pushing forward a twisted vector bundle. We provide two classes of examples satisfying $\Br \ne \Br'$ that do not ``arise from'' the scheme case. \end{abstract}
\date{\today}
\maketitle
\tableofcontents

\section{Introduction}

For an algebraic stack $\ms{X}$, let $\Br \ms{X}$ denote the \emph{Brauer group} of $\ms{X}$, namely the set of Brauer equivalence classes of Azumaya $\mc{O}_{\ms{X}}$-algebras. Grothendieck \cite{Gro68b} defined a functorial injective homomorphism $\alpha_{\ms{X}}' : \Br \ms{X} \to \H_{\fppf}^{2}(\ms{X},\G_{m})$, which sends an Azumaya $\mc{O}_{\ms{X}}$-algebra $\mc{A}$ to the $\G_{m}$-gerbe of trivializations of $\mc{A}$; the image of $\alpha_{\ms{X}}'$ is contained in the \emph{cohomological Brauer group} of $\ms{X}$, namely the subgroup $\Br' \ms{X} := \H_{\fppf}^{2}(\ms{X},\G_{m})_{\mr{tors}}$ of torsion classes. The restriction \begin{align} \label{20160903-27-eqn-03} \alpha_{\ms{X}} : \Br \ms{X} \to \Br' \ms{X} \end{align} is called the \emph{Brauer map}. If $\alpha_{\ms{X}}$ is an isomorphism, we also say ``$\Br = \Br'$ for $\ms{X}$''.

In general, it is difficult to determine whether $\Br = \Br'$ for a given $\ms{X}$. At the moment, the most widely-cited affirmative result is a theorem of Gabber \cite{DEJONG-GABBER} that $\Br = \Br'$ for any scheme which admits an ample line bundle.  The only known class of counterexamples are non-separated schemes \cite{EdHaKrVi}; it remains open whether $\Br = \Br'$ for all smooth varieties over a field.

In this paper, we are concerned with the $\Br = \Br'$ question for certain classifying stacks. We propose to measure the worth of a morphism according to the following definition:

\begin{definition}[``surjectivity of Brauer map invariant''] \label{20171113-56} Let $\ms{X} \to \ms{Y}$ be a morphism of algebraic stacks. We say that $\ms{X} \to \ms{Y}$ satisfies $(\mr{SBMI})$ if for all $\ms{Y}$-schemes $S$ we have $\Br = \Br'$ for $S$ if and only if $\Br = \Br'$ for $\ms{X} \times_{\ms{Y}} S$. \end{definition}

\begin{remark} If $\ms{X} \to \ms{Y}$ admits a section and $\Br'(S) \to \Br'(\ms{X} \times_{\ms{Y}} S)$ is an isomorphism for all $S \to \ms{Y}$, then $\ms{X} \to \ms{Y}$ is $(\mr{SBMI})$. For example, any projective space $\P_{\Z}^{n} \to \Spec \Z$ is $(\mr{SBMI})$. \end{remark}

Our main results are the following:

\begin{theorem} \label{main-theorem} Let $S$ be a base scheme. \begin{enumerate} \item For a finitely generated abelian group $G$, the map $\mr{B}G_{\Spec \Z} \to \Spec \Z$ is $(\mr{SBMI})$ if and only if $G$ has rank 1 (\Cref{20171113-57}). \item If $S = \Spec k$ for a field $k$ and $A$ is an abelian variety over $k$, then $\Br = \Br'$ for $\mr{B}A$ if and only if $\Pic_{A/k}^{0}(k)$ is torsion-free (\Cref{20170215-16}). \item Let $D \to \Spec \Z$ be a diagonalizable group scheme. Then $\mr{B}D \to \Spec \Z$ is $(\mr{SBMI})$ (\Cref{20170920-03}). \item Suppose $S$ is a Noetherian normal scheme. Then $\Br = \Br'$ for $S$ if and only if $\Br = \Br'$ for $\mr{B}\GL_{n,S}$ (\Cref{20170920-18}). \end{enumerate} \end{theorem}

\begin{pg} \label{0001} Part (4) in \Cref{main-theorem} says that the map $\mr{B}\GL_{n,S} \to S$ is ``almost $(\mr{SBMI})$''. The ``normal'' hypothesis allows us to show that the pullback morphism $\Br'(S) \to \Br'(\mr{B}\GL_{n,S})$ is an isomorphism. In general, in order to compute the cohomological Brauer groups of classifying stacks $\mr{B}G$, it is useful to be able to compute cohomology of products $G^{\times p}$, but we encounter difficulties because $\H_{\fppf}^{i}(-,\G_{m})$ does not behave well under even polynomial extensions. More precisely, we can ask for ring-theoretic properties of $A$ which imply that the injection \begin{align} \label{0001-01} \H_{\fppf}^{i}(\Spec A,\G_{m}) \to \H_{\fppf}^{i}(\Spec A[t],\G_{m}) \end{align} is an isomorphism (i.e. whether $\H_{\fppf}^{i}(-,\G_{m})$ is ``$\A^{1}$-homotopy invariant''). For $i=0$ (resp. $i=1$), it is known that \labelcref{0001-01} is an isomorphism if and only if $A$ is reduced (resp. $A$ is seminormal); for $i=2$, it is known that $\Br'(A) \to \Br'(A[t])$ is an isomorphism if $A$ either contains $\Q$ or is regular with perfect fraction field. \end{pg}

\begin{pg} A standard argument shows that $\Br = \Br'$ for quotient stacks $\ms{X} = [U/G]$ where $U$ is a scheme for which $\Br = \Br'$ and $G$ is a finite flat group scheme. Since every separated Deligne-Mumford stack $\ms{X}$ is of this form etale-locally on its coarse moduli space $X$, for every Brauer class $c \in \Br'(\ms{X})$ there exists an etale cover $X' \to X$ such that $c|_{\ms{X} \times_{X} X'}$ is contained in $\Br(\ms{X} \times_{X} X')$. (For algebraic stacks with stabilizer groups that are either positive-dimensional or not quasi-compact, this argument does not apply and we know of no substitute, so we were led to consider various examples as in \Cref{main-theorem}.) In \Cref{sec-linred} we explain a result of Siddharth Mathur which is an analogue of the above local existence result for algebraic stacks admitting a good moduli space; this allows us to compute $\Br'(\mr{B}D_{X})$ for a non-normal scheme $X$ in part (3) of \Cref{main-theorem}. (Mathur has also recently proved global results \cite[Theorem 2]{MATHUR-TRPVAA2020}, namely, for a tame algebraic stack $\ms{X}$ over a field $k$, we have $\Br = \Br'$ for $\ms{X}$ if either (1) $\dim \ms{X} = 1$ or (2) $\dim \ms{X} = 2$, regular in codimension 1, and a gerbe over a stack with generically trivial stabilizers. This may be viewed as a stacky generalization of Grothendieck's first results \cite[\S2]{Gro68b} on $\Br = \Br'$ for curves and smooth surfaces.) \end{pg}

\begin{pg}[Acknowledgements] Most of the material in this paper is part of my thesis \cite{SHIN-THESIS2019}, which was completed under the supervision of Martin Olsson. I am also grateful to Jarod Alper, Max Lieblich, Siddharth Mathur, and Stefan Schr\"oer for helpful conversations. Part of this project was completed during my visit to the Max Planck Institute for Mathematics. \end{pg}

\section{Generalities}

The following are well-known:

\begin{lemma} \label{20160630-wr-01} Let $f : \ms{X} \to \ms{Y}$ be a finite, flat, finitely presented, surjective morphism of algebraic stacks. If $c \in \Br' \ms{Y}$ is an element such that $f^{\ast}c \in \im \alpha_{\ms{X}}$ then $c \in \im \alpha_{\ms{Y}}$. \end{lemma}

\begin{lemma} \label{20160630-wr-02} Let $f : \ms{X} \to \ms{Y}$ be a morphism of algebraic stacks admitting a section $s : \ms{Y} \to \ms{X}$. If $c \in \Br' \ms{Y}$ is an element such that $f^{\ast}c \in \im \alpha_{\ms{X}}$ then $c \in \im \alpha_{\ms{Y}}$. \end{lemma} \begin{proof} We have a commutative diagram \begin{center}\begin{tikzpicture}[>=stealth]
\matrix[matrix of math nodes,row sep=2em, column sep=2em, text height=2ex, text depth=0.25ex] { 
|[name=11]| \Br \ms{Y} & |[name=12]| \Br \ms{X} & |[name=13]| \Br \ms{Y} \\ 
|[name=21]| \Br' \ms{Y} & |[name=22]| \Br' \ms{X} & |[name=23]| \Br' \ms{Y} \\
}; 
\draw[->,font=\scriptsize] (11) edge node[above=-1pt] {$f^{\ast}$} (12) (12) edge node[above=-1pt] {$s^{\ast}$} (13) (21) edge node[below=-1pt] {$f^{\ast}$} (22) (22) edge node[below=-1pt] {$s^{\ast}$} (23) (11) edge node[left=-1pt] {$\alpha_{\ms{Y}}$} (21) (12) edge node[left=-1pt] {$\alpha_{\ms{X}}$} (22) (13) edge node[left=-1pt] {$\alpha_{\ms{Y}}$} (23); \end{tikzpicture} \end{center} where the horizontal arrows compose to the identity. If $f^{\ast}c = \alpha_{\ms{X}}(\beta)$ for some $\beta \in \Br \ms{X}$, then $c = s^{\ast}f^{\ast}c = s^{\ast}\alpha_{\ms{X}}(\beta) = \alpha_{\ms{Y}}(s^{\ast}\beta)$. \end{proof}

\begin{corollary} \label{20160630-wr-11} Let $\ms{X}$ be a smooth separated generically tame Deligne-Mumford stack over a field $k$ with quasi-projective coarse moduli space. Then $\Br = \Br'$ for $\ms{X}$. \end{corollary} \begin{proof} By Kresch and Vistoli \cite[2.1,2.2]{KV2004}, such $\ms{X}$ has a finite flat surjection $Z \to \ms{X}$ where $Z$ is a quasi-projective $k$-scheme. By Gabber's theorem \cite{DEJONG-GABBER}, the Brauer map is surjective for $Z$. Thus the Brauer map is surjective for $\ms{X}$ by \Cref{20160630-wr-01}. \end{proof}

\begin{lemma} \label{20160630-wr-03} Let $1 \to G_{1} \to G_{2} \to G_{3} \to 1$ be an exact sequence of group schemes over $S$ such that $G_{3} \to S$ is finite flat. If a Brauer class $c \in \Br'(\mr{B}G_{2})$ is such that $c|_{\mr{B}G_{1}}$ is contained in $\Br(\mr{B}G_{1})$, then $c$ is contained in $\Br(\mr{B}G_{2})$. \end{lemma} \begin{proof} This follows from \Cref{20160630-wr-01}, using that $\mr{B}G_{1} \to \mr{B}G_{2}$ is a $G_{3}$-torsor. \end{proof}

\begin{pg} \label{20151211-04} A key computational tool for us is the descent spectral sequence (see e.g. \cite[\S2.4]{OLSSON}). Let $\ms{X}$ be an algebraic stack, let $X$ be an algebraic space with a smooth surjection $X \to \ms{X}$. Let $X^{p}$ denote the $(p+1)$-fold 2-fiber product $X \times_{\ms{X}} \dotsb \times_{\ms{X}} X$; for any abelian sheaf $\mr{A}$ on $\ms{X}$, there is a spectral sequence \[ \mr{E}_{1}^{p,q} = \H^{q}(X^{p},\mr{A}) \implies \H^{p+q}(\ms{X},\mr{A}) \] with differentials $\mr{E}_{1}^{p,q} \to \mr{E}_{1}^{p+1,q}$, where the $q$th row $\{\mr{E}_{1}^{\bullet,q}\}$ is the Cech complex obtained by applying the functor $\H^{q}(-,\mr{A})$ to the simplicial algebraic space $\{X^{\bullet}\}$. In case $\ms{X} = [X/G]$ for a discrete group $G$, we have in fact $X^{p} \simeq X \times G^{p}$ and the next page of the above spectral sequence is the usual Hochschild-Serre spectral sequence \[ \mr{E}_{2}^{p,q} = \H^{p}(G,\H^{q}(X,\mr{A})) \implies \H^{p+q}(\ms{X},\mr{A}) \] with differentials $\mr{E}_{2}^{p,q} \to \mr{E}_{2}^{p+2,q-1}$. \end{pg}

\section{The classifying stack of discrete groups} \label{sec-discrete}

\begin{pg} By \Cref{20160630-wr-01}, it is easy to see that if $G$ is a finite flat $S$-group scheme, then $\mr{B}G_{S} \to S$ satisfies $(\mr{SBMI})$. In this section, we are concerned with the case when $G$ is the discrete group scheme associated to a (possibly infinite) finitely generated abelian group. \end{pg}

\begin{lemma} \label{20171113-15} The classifying stack $\mr{B}\Z_{\Spec \Z} \to \Spec \Z$ satisfies $(\mr{SBMI})$. \end{lemma} \begin{proof} Let $S$ be a scheme, let $\pi : \mr{B}\Z_{S} \to S$ be the projection and let $s : S \to \mr{B}\Z_{S}$ be the section corresponding to the trivial torsor. The spectral sequence \Cref{20151211-04} associated to the covering $s$ is \[ \mr{E}_{2}^{p,q} = \H^{p}(\Z,\H^{q}_{\fppf}(S,\G_{m})) \implies \H^{p+q}_{\fppf}(\mr{B}\Z_{S},\G_{m}) \] with differentials $\mr{E}_{2}^{p,q} \to \mr{E}_{2}^{p+2,q-1}$. Here $\mr{E}_{2}^{p,q} = 0$ for $p \ge 2$ (using that group cohomology for $G = \Z$ is supported in degrees $0,1$), thus there is a short exact sequence \[ 0 \to \Pic(S) \to \H_{\fppf}^{2}(\mr{B}\Z_{S},\G_{m}) \to \H_{\fppf}^{2}(S,\G_{m}) \to 0 \] of abelian groups, which is split. Thus we have a commutative diagram \begin{equation} \label{20171113-15-eqn-01} \begin{tikzpicture}[>=stealth, baseline=(current bounding box.center)] 
\matrix[matrix of math nodes,row sep=2.5em, column sep=2em, text height=1.5ex, text depth=0.25ex] { 
|[name=11]| 0 & |[name=12]| \ker(s^{\ast}) & |[name=13]| \Br(\mr{B}\Z_{S}) & |[name=14]| \Br(S) & |[name=15]| 0 \\ 
|[name=21]| 0 & |[name=22]| \Pic(S) & |[name=23]| \H_{\fppf}^{2}(\mr{B}\Z_{S},\G_{m}) & |[name=24]| \H_{\fppf}^{2}(S,\G_{m}) & |[name=25]| 0 \\ 
}; 
\draw[->,font=\scriptsize] (11) edge (12) (12) edge node[above=0pt] {$g$} (13) (13) edge node[above=0pt] {$s^{\ast}$} (14) (14) edge (15) (21) edge (22) (22) edge node[below=0pt] {$g'$} (23) (23) edge node[below=0pt] {$(s^{\ast})'$} (24) (24) edge (25) (12) edge node[right=-1pt] {$f_{1}'$} (22) (13) edge node[right=-1pt] {$\alpha_{\mr{B}\Z_{S}}'$} (23) (14) edge node[right=-1pt] {$\alpha_{S}'$} (24); \end{tikzpicture} \end{equation} with exact rows; the pullback $\pi$ induces splitting of both rows that are compatible with $\alpha_{\mr{B}\Z_{S}}'$ and $\alpha_{S}'$. Taking the torsion parts gives a commutative diagram \begin{equation} \label{20171113-15-eqn-02} \begin{tikzpicture}[>=stealth, baseline=(current bounding box.center)] 
\matrix[matrix of math nodes,row sep=2.5em, column sep=2em, text height=1.5ex, text depth=0.25ex] { 
|[name=11]| 0 & |[name=12]| \ker(s^{\ast}) & |[name=13]| \Br(\mr{B}\Z_{S}) & |[name=14]| \Br(S) & |[name=15]| 0 \\ 
|[name=21]| 0 & |[name=22]| \Pic(S)_{\mr{tors}} & |[name=23]| \Br'(\mr{B}\Z_{S}) & |[name=24]| \Br'(S) & |[name=25]| 0 \\ 
}; 
\draw[->,font=\scriptsize] (11) edge (12) (12) edge node[above=0pt] {$g$} (13) (13) edge node[above=0pt] {$s^{\ast}$} (14) (14) edge (15) (21) edge (22) (22) edge node[below=0pt] {$g'$} (23) (23) edge node[below=0pt] {$(s^{\ast})'$} (24) (24) edge (25) (12) edge node[right=-1pt] {$f_{1}$} (22) (13) edge node[right=-1pt] {$\alpha_{\mr{B}\Z_{S}}$} (23) (14) edge node[right=-1pt] {$\alpha_{S}$} (24); \end{tikzpicture} \end{equation} since the bottom row of \labelcref{20171113-15-eqn-01} is split. \par From diagram \labelcref{20171113-15-eqn-02}, we have that if $\alpha_{\mr{B}\Z_{S}}$ is an isomorphism, then $\alpha_{S}$ is an isomorphism. \par For the converse, it suffices to show that $f_{1}$ is an isomorphism. \par We first describe $\ker(s^{\ast})$. An Azumaya $\mc{O}_{\mr{B}\Z_{S}}$-algebra that becomes trivial after forgetting the $\Z$-action corresponds to a pair \[ (\mc{E},c) \] where $\mc{E}$ is a finite locally free $\mc{O}_{S}$-module and $c$ is an element of $\Aut_{\mc{O}_{S}\text{-alg}}(\mc{A})$ where $\mc{A} := \underline{\End}_{\mc{O}_{S}\text{-mod}}(\mc{E})$. We have an exact sequence \begin{align} \label{20171113-15-eqn-03} 1 \to \G_{m,S} \to \mc{A}^{\times} \to \underline{\Aut}_{\mc{O}_{S}\text{-alg}}(\mc{A}) \to 1 \end{align} of groups on the etale site of $S$; the image of $(\mc{E},c)$ under $f_{1}$ is the image of $c$ under the coboundary map $\Aut_{\mc{O}_{S}\text{-alg}}(\mc{A}) \to \Pic(S)$ associated to \labelcref{20171113-15-eqn-03}. The map $f_{1}$ is injective by commutativity of \labelcref{20171113-15-eqn-02}; in words, if the image of $c$ in $\Pic(S)$ is trivial, then $c$ is the automorphism of $\mc{A}$ obtained by conjugation by an automorphism of $\mc{E}$. \par We can check that the image of this map lands in the $(\rank \mc{E})$-torsion part of $\Pic(S)$ as follows. Let $c \in \Aut_{\mc{O}_{S}\text{-alg}}(\mc{A})$ be an algebra automorphism. Let $n = \rank \mc{E}$ be the rank of $\mc{E}$. Given an $S$-scheme $T \to S$ and two module automorphisms $c_{1},c_{2} \in \Aut_{\mc{O}_{S}\text{-mod}}(\mc{E})$ such that $c_{T}$ corresponds to conjugation by $c_{1},c_{2}$, there exists some $u \in \G_{m}(T)$ such that $c_{1} = u c_{2}$. Then taking determinants gives $\det c_{1} = u^{n} \det c_{2}$. \par We show that $f_{1}$ is surjective. Let $\mc{L}$ be an invertible $\mc{O}_{S}$-module which is $n$-torsion. Let $\mf{U} = \{U_{i}\}_{i \in I}$ be a trivializing cover for $\mc{L}$, and choose sections \[ s_{i} \in \Gamma(U_{i},\mc{L}) \] which trivialize $\mc{L}|_{U_{i}}$; for all $(i_{1},i_{2}) \in I \times I$, let \[ \xi_{i_{1},i_{2}} \in \Gamma(U_{i_{1}} \cap U_{i_{2}},\G_{m}) \] be the unique section such that \[ \xi_{i_{1},i_{2}}s_{i_{1}} = s_{i_{2}} \] in $\Gamma(U_{i_{1}} \cap U_{i_{2}} , \mc{L})$. The collection $\{\xi_{i_{1},i_{2}}\}$ is a 1-cocycle corresponding to $\mc{L}$. The condition that $\mc{L}$ is $n$-torsion implies that there are sections \[ \beta_{i} \in \Gamma(U_{i},\G_{m}) \] such that \[ \xi_{i_{1},i_{2}}^{n} = \beta_{i_{1}}^{-1} \beta_{i_{2}} \] for all $(i_{1},i_{2}) \in I \times I$. Let $\ml{E}_{i_{1},i_{2}}$ be the $n \times n$ matrix whose $(i_{1},i_{2})$th entry is $1$ and all other entries are $0$. We take \[ \textstyle \mc{E} := \bigoplus_{\ell=0}^{n-1} \mc{L}^{\otimes \ell} \] and describe an $\mc{O}_{S}$-algebra automorphism $c : \underline{\End}_{\mc{O}_{S}}(\mc{E}) \to \underline{\End}_{\mc{O}_{S}}(\mc{E})$ such that the pair $(\mc{E},c)$ maps to the class of $\mc{L}$ in $\Pic(S)$. We describe $c$ on the restrictions to $U_{i}$ and show that it glues. Let \[ \varphi_{i} : \mc{E}|_{U_{i}} \to \mc{E}|_{U_{i}} \] be the $\mc{O}_{U_{i}}$-module automorphism acting by the matrix \[ \ml{M}_{i} := \ml{E}_{2,1} + \dotsb + \ml{E}_{n,n-1} + \beta_{i} \ml{E}_{1,n} \] with respect to the basis $(1,s_{i},\dotsc,s_{i}^{\otimes n-1})$, and let \[ c_{i} : \underline{\End}_{\mc{O}_{U_{i}}}(\mc{E}|_{U_{i}}) \to \underline{\End}_{\mc{O}_{U_{i}}}(\mc{E}|_{U_{i}}) \] be the conjugation-by-$\varphi_{i}$ map. On the intersections $U_{i_{1}} \cap U_{i_{2}}$, the matrix corresponding to the change of basis $(1,s_{i_{1}},\dotsc,s_{i_{1}}^{\otimes n-1}) \to (1,s_{i_{2}},\dotsc,s_{i_{2}}^{\otimes n-1})$ is \[ \textstyle \ml{D}_{i_{1},i_{2}} := \sum_{\ell=0}^{n-1} \xi_{i_{1},i_{2}}^{\ell}\ml{E}_{\ell,\ell} \] and the condition that $c_{i_{1}},c_{i_{2}}$ agree on $U_{i_{1}} \cap U_{i_{2}}$ amounts to the claim that $\ml{D}_{i_{1},i_{2}}\ml{M}_{i_{1}}$ and $\ml{M}_{i_{2}}\ml{D}_{i_{1},i_{2}}$ differ by an element of $u\ml{id}_{n}$ for some $u \in \Gamma(U_{i_{1}} \cap U_{i_{2}} , \G_{m})$. We have in fact \[ \ml{D}_{i_{1},i_{2}}\ml{M}_{i_{1}} = \xi_{i_{1},i_{2}}\ml{M}_{i_{2}}\ml{D}_{i_{1},i_{2}} \] which shows that the collection $\{c_{i}\}_{i \in I}$ glues to give a global automorphism $c : \underline{\End}_{\mc{O}_{S}}(\mc{E}) \to \underline{\End}_{\mc{O}_{S}}(\mc{E})$ and that the image of the pair $(\mc{E},c)$ in $\Pic(S)$ is $\mc{L}$. \end{proof}

\begin{lemma}[Group cohomology for $\Z \oplus \Z$] \label{20150528-09} \footnote{We include this result (which was first posted as an answer at \url{https://math.stackexchange.com/q/2611736/}) for lack of awareness of an appropriate reference. It is likely possible to give an alternate proof using Kunneth formulas for group cohomology \cite[V, \S2]{BROWN-COG1982}, \cite[Exercise 6.1.10]{WEIBEL}.} Set $G := \Z \oplus \Z$, let $M$ be a $G$-module, let $A := \Z[t_{1}^{\pm},t_{2}^{\pm}]$ be the group ring of $G$. Then we have \[ \mathrm{H}^{2}(G, M) \simeq \coker(M^{\oplus 2} \stackrel{f_{2}^{\ast}}{\to} M) \] where the map $f_{2}^{\ast} : M^{\oplus 2} \to M$ sends $(m_{1},m_{2}) \mapsto (t_{2}-1)m_{1} - (t_{1}-1)m_{2}$. \end{lemma} \begin{proof} The $A$-module \[ \Z \simeq A/(t_{1}-1,t_{2}-1)A \] has an $A$-module resolution \[ 0 \to A e_{2,1} \stackrel{f_{2}}{\to} A e_{1,1} \oplus A e_{1,2} \stackrel{f_{1}}{\to} A e_{0,1} \to \Z \to 0 \] where $f_{2}(e_{2,1}) = (t_{2}-1)e_{1,1} - (t_{1}-1)e_{1,2}$ and $f_{1}(e_{1,i}) = (t_{i}-1)e_{0,1}$ for $i=1,2$. Applying $\Hom_{A}(-,M)$ to the above resolution gives a complex \[ M \stackrel{f_{1}^{\ast}}{\to} M^{\oplus 2} \stackrel{f_{2}^{\ast}}{\to} M \to 0 \to \dotsb \] of $A$-modules, and taking cohomology at the $i$th cohomological degree gives $\mathrm{H}^{i}(G,M)$. \end{proof}

\begin{example} \label{20171113-14} Here we discuss a class of regular Deligne-Mumford stacks $\ms{X}$ where $\Br(\ms{X}) \to \Br'(\ms{X})$ is not an isomorphism. Such $\ms{X}$ can be modified to have any dimension but the diagonal morphism is not quasi-compact (in particular, not affine). \par Let $A$ be a ring for which all vector bundles are trivial (e.g a semi-local ring or a polynomial ring over a PID), set $S := \Spec A$, and let $G := \Z \oplus \Z$. We view $G$ as acting trivially on $A$. Let $\ms{X} := [S/G] \simeq \mr{B}G_{S}$ be the classifying stack. We have the cohomological descent spectral sequence \[ \mr{E}_{2}^{p,q} = \H^{p}(G,\H^{q}_{\fppf}(S,\G_{m,S})) \implies \H^{p+q}_{\fppf}(\ms{X},\G_{m,\ms{X}}) \] with differentials $\mr{E}_{2}^{p,q} \to \mr{E}_{2}^{p+2,q-1}$. We have $\H_{\fppf}^{1}(S,\G_{m,S}) = \Pic(S) = 0$, and furthermore $\mr{E}_{2}^{p,q} = 0$ if $p \ge 3$ by \Cref{20150528-09}, thus we have an direct sum decomposition \[ \H^{2}_{\fppf}(\ms{X},\G_{m,\ms{X}}) = \H_{\fppf}^{2}(S,\G_{m,S}) \oplus \H^{2}(G,A^{\times}) \] of abelian groups (a priori only an exact sequence but it is split as the projection $\pi : \ms{X} \to S$ has a section $s : S \to \ms{X}$). We have a direct sum decomposition $\Br(\ms{X}) = \Br(A) \oplus \ker(s^{\ast} : \Br(\ms{X}) \to \Br(A))$. \par An element of $\ker(s^{\ast} : \Br(\ms{X}) \to \Br(A))$ corresponds to an Azumaya $\mc{O}_{\ms{X}}$-algebra $\mc{A}$ such that $s^{\ast}\mc{A}$ is a trivial Azumaya $A$-algebra; this corresponds to a group homomorphism $G \to \PGL_{r}(A)$ where $\mc{A}$ has rank $r^{2}$. A vector bundle on $\ms{X}$ of rank $r$ corresponds to a group homomorphism $G \to \GL_{r}(A)$. Since $\Pic(A) = 0$, the map $\GL_{r}(A) \to \PGL_{r}(A)$ is surjective. Since $G$ is a free abelian group, the map $\H^{1}(G,\GL_{r}(A)) \to \H^{1}(G,\PGL_{r}(A))$ is surjective. Thus such $\mc{A}$ is trivial, in other words the pullback $\pi^{\ast} : \Br(A) \to \Br(\ms{X})$ is an isomorphism. \par On the other hand, we have $\H^{2}(G,A^{\times}) = A^{\times}$ by \Cref{20150528-09}, thus $\Br'(\ms{X}) = \Br'(A) \oplus (A^{\times})_{\mr{tors}}$. There are regular local rings $A$ such that $A^{\times}$ has a lot of torsion (take a local ring of a smooth $k$-scheme where $k$ is an algebraically closed field of characteristic $0$, for example). \qed \end{example}

\begin{theorem} \label{20171113-57} Let $G$ be a finitely generated abelian group. The classifying stack $\mr{B}G_{\Spec \Z} \to \Spec \Z$ satisfies $(\mr{SBMI})$ if and only if $G$ has rank $1$ (i.e. $\dim_{\Q}(G \otimes_{\Z}\Q) = 1$). \end{theorem} \begin{proof} By \Cref{20160630-wr-02}, if $\Br = \Br'$ for $\mr{B}G_{S}$ then $\Br = \Br'$ for $S$. We may write $G \simeq H \oplus \Z^{r}$ for some finite abelian group $H$. If $r \ge 2$, then there are group homomorphisms $\Z^{2} \to G \to \Z^{2}$ whose composition is the identity. As in \Cref{20171113-14}, there exists an affine scheme $S$ and a class $c \in \Br'(\mr{B}\Z_{S}^{2})$ which is not contained in $\im \alpha_{\mr{B}\Z_{S}^{2}}$. This means that $c|_{\mr{B}G_{S}}$ is also not contained in $\im \alpha_{\mr{B}G_{S}}$. \par Suppose $r \le 1$ and that $\Br = \Br'$ for $S$, and let $c \in \Br'(\mr{B}G_{S})$. We have a map $\Z \to G$ with cokernel $H$; by \Cref{20171113-15}, we have that $c|_{\mr{B}\Z_{S}}$ is contained in $\im \alpha_{\mr{B}\Z_{S}}$, hence $c$ is contained in $\im \alpha_{\mr{B}G_{S}}$ by \Cref{20160630-wr-03}. \end{proof}

\section{The classifying stack of abelian varieties} \label{sec-abvar}

In this section, we show that classifying stacks of abelian varieties provide another class of algebraic stacks for which $\Br \ne \Br'$. These stacks have do not have affine diagonal.

\begin{lemma} \label{20170215-13} Let $k$ be a field, let $A$ be an abelian variety over $k$, let $m : A \times_{k} A \to A$ be the group law and let $p_{1},p_{2} : A \times_{k} A \to A$ be the projections. Let us denote by \[ m^{\ast}-p_{1}^{\ast}-p_{2}^{\ast} : \Pic(A) \to \Pic(A \times_{k} A) \] the map sending a line bundle on $A$ to the associated ``Mumford bundle''. Then the group of ``translation-invariant'' line bundles is \[ \ker(m^{\ast}-p_{1}^{\ast}-p_{2}^{\ast}) = \Pic_{A/k}^{0}(k) \;. \]  \end{lemma} \begin{proof} Given a line bundle $\mc{L}$ on $A$, the Mumford bundle of $\mc{L}$ is \[ \Lambda(\mc{L}) := m^{\ast}\mc{L} - p_{1}^{\ast}\mc{L} - p_{2}^{\ast}\mc{L} \] on $A \times_{k} A$; viewing $A \times_{k} A$ as an $A$-scheme with structure map $p_{1}$ gives a morphism \[ \phi_{\mc{L}} : A \to \Pic_{A/k} \] corresponding to $\Lambda(\mc{L})$; here $\phi_{\mc{L}}$ factors through the dual abelian variety $A^{t} = \Pic_{A/k}^{0}$. The assignment $\mc{L} \mapsto \phi_{\mc{L}}$ gives a morphism \[ \phi : \Pic_{A/k} \to \underline{\Hom}(A,A^{t}) \] of group sheaves on $k$. By e.g. \cite[(7.22) Corollary]{EDIXHOVEN-MOONEN-VANDERGEER-ABELIAN-VARIETIES}, we have \[ \Pic_{A/k}^{0} = \ker \phi \] which gives the desired result. \end{proof}

\begin{lemma} \label{20170215-14} Let $k$ be a field and let $A$ be an abelian variety over $k$. We have an isomorphism \[ \H_{\fppf}^{2}(\mr{B}A , \G_{m}) \simeq \Br(k) \oplus \Pic_{A/k}^{0}(k) \] of groups. \end{lemma} \begin{proof} We compute $\H_{\fppf}^{2}(\mr{B}A , \G_{m})$ using the cohomological descent spectral sequence \[ \mr{E}_{1}^{p,q} = \H_{\fppf}^{q}(A^{\times p} , \G_{m}) \implies \H_{\fppf}^{p+q}(\mr{B}A,\G_{m}) \] with differentials $\mr{E}_{1}^{p,q} \to \mr{E}_{1}^{p+1,q}$. We have $\H_{\fppf}^{0}(A^{\times p} , \G_{m}) = k$ for all $p$, and the complex $\H_{\fppf}^{0}(A^{\times \bullet} , \G_{m})$ is acyclic except at $p = 0$. The map $\mr{E}_{\infty}^{2} \to \mr{E}_{1}^{0,2}$ corresponds to the pullback $\H_{\fppf}^{2}(\mr{B}A , \G_{m}) \to \H_{\fppf}^{2}(\Spec k , \G_{m})$, which is a split surjection since the composite $\Spec k \to \mr{B}A \to \Spec k$ is the identity. We have $\mr{E}_{2}^{1,1} \simeq \Pic_{A/k}^{0}(k)$ by \Cref{20170215-13}. \end{proof}

\begin{lemma} \label{20170215-15} Let $k$ be a field and let $A$ be an abelian variety over $k$. Then the pullback morphism $\Br(k) \to \Br(\mr{B}A)$ is an isomorphism. \end{lemma} \begin{proof} Let $\xi : \Spec k \to \mr{B}A$ be the morphism corresponding to the trivial $A$-torsor. There is a direct sum decomposition $\Br(\mr{B}A) = \Br(k) \oplus \ker(\xi^{\ast})$ where $\xi^{\ast} : \Br(\mr{B}A) \to \Br(k)$ is the pullback map. A class in $\ker(\xi^{\ast})$ corresponds to an Azumaya $\mc{O}_{\mr{B}A}$-algebra $\mc{A}$ which is trivialized after pullback by $\xi$; this is the data of a positive integer $n$ and an element $\varphi \in \PGL_{n}(A)$ which satisfies the cocycle condition on $A \times_{k} A$, more precisely $m^{\ast}\varphi = p_{1}^{\ast}\varphi \cdot p_{2}^{\ast}\varphi$ where $m,p_{1},p_{2}$ are as in \Cref{20170215-13}. Since $\PGL_{n}$ is affine, the pullback $\PGL_{n}(\Gamma(A,\mc{O}_{A})) \to \PGL_{n}(A)$ is an isomorphism; similarly $\PGL_{n}(\Gamma(A,\mc{O}_{A})) \to \PGL_{n}(A \times_{k} A)$ is an isomorphism as well. Since $A$ is geometrically integral, we have $k \to \Gamma(A,\mc{O}_{A})$ and $k \to \Gamma(A \times_{k} A,\mc{O}_{A \times_{k} A})$ are isomorphisms. Thus $\varphi$ is an element of $\PGL_{n}(k)$ which satisfies $\varphi = \varphi \cdot \varphi$, in other words $\varphi = \id$. Thus $\mc{A}$ is isomorphic to $\mr{Mat}_{n \times n}(\mc{O}_{\mr{B}A})$. \end{proof}

\begin{theorem} \label{20170215-16} Let $k$ be a field and let $A$ be an abelian variety over $k$. Then $\Br = \Br'$ for $\mr{B}A$ if and only if $\Pic_{A/k}^{0}(k)$ is torsion-free. \end{theorem} \begin{proof} This follows from \Cref{20170215-14} and \Cref{20170215-15}. \end{proof}

\section{Quotient stacks by linearly reductive group schemes} \label{sec-linred}

Suppose $\ms{X}$ is a separated Deligne-Mumford stack and let $\ms{X} \to X$ be its coarse moduli space; every geometric point $\overline{x} \to X$ admits an etale neighborhood $U \to X$ such that $\ms{X} \times_{X} U$ is a quotient stack of an affine scheme by a finite group \cite[2.2.3]{ABRAMOVICH-VISTOLI-COMPACTIFYING}, hence $\Br = \Br'$ for $\ms{X} \times_{X} U$ by \Cref{20160630-wr-01}.

In this section, we explain an analogue of this fact for algebraic stacks with higher-dimensional stabilizers. This argument is due to Siddharth Mathur \cite{MATHUR-PRIVATECORRESPONDENCE-20190912-01}. 

\begin{lemma} \label{20190128-12} Let $X$ be a Noetherian algebraic stack such that $X$ has affine diagonal and there exists a good moduli space morphism $X \to S$. Let $\ms{X} \to X$ be a $\G_{m}$-gerbe. Suppose that, for every geometric point $\overline{s} \to S$, the fiber $\ms{X}_{\overline{s}} \to X_{\overline{s}}$ admits a 1-twisted vector bundle of rank $r$. Then there exists an etale cover $S' \to S$ such that $\ms{X}_{S'}$ admits a 1-twisted vector bundle of rank $r$. \end{lemma} \begin{proof} We note that the composition $\ms{X} \to X \to S$ is a good moduli space by \cite[4.1]{ALPER-GMSFAS} and \cite[3.10 (i)]{ALPER-GMSFAS}. The stack $\ms{X}$ is an algebraic stack \cite[06PL]{SP} and its diagonal $\Delta_{\ms{X}/S}$ is affine because it is the composite of the two upper arrows in the following diagram: \begin{center} \begin{tikzpicture}[>=angle 90] 
\matrix[matrix of math nodes,row sep=2em, column sep=2em, text height=1.7ex, text depth=0.5ex] { 
|[name=11]| \ms{X} & |[name=12]| \ms{X} \times_{X} \ms{X} & |[name=13]| \ms{X} \times_{S} \ms{X} \\ 
|[name=21]|  & |[name=22]| X & |[name=23]| X \times_{S} X \\
}; 
\draw[->,font=\scriptsize] (11) edge node[above=0pt] {$\Delta_{\ms{X}/X}$} (12) (12) edge (13) (12) edge (22) (13) edge (23) (22) edge node[below=0pt] {$\Delta_{X/S}$} (23); \end{tikzpicture} \end{center} By standard limit arguments (and since good moduli space morphisms are preserved under arbitrary base change \cite[4.7 (i)]{ALPER-GMSFAS}), we may assume that $S = \Spec A$ for a strictly henselian local ring $A$ with maximal ideal $\mf{m}$ and residue field $k := A/\mf{m}$. \par For $n \in \N$, set $X_{n} := X \times_{\Spec A} \Spec A/\mf{m}^{n+1}$ and $\ms{X}_{n} := \ms{X} \times_{\Spec A} \Spec A/\mf{m}^{n+1}$. Given a 1-twisted vector bundle $\ms{E}_{n}$ of rank $r$ on $\ms{X}_{n}$, the obstruction to lifting $\ms{E}_{n}$ to a 1-twisted vector bundle $\ms{E}_{n+1}$ of rank $r$ on $\ms{X}_{n+1}$ is contained in $\H^{2}(\ms{X}_{n},\mf{m}^{n}\mc{O}_{\ms{X}_{n}}^{\oplus r^{2}})$, which vanishes by \cite[3.5]{ALPER-GMSFAS}. Let $A^{\wedge}$ denote the completion of $A$; by formal GAGA for good moduli space morphisms \cite[Corollary 1.7]{ALPERHALLRYDH-TELSOAS2019}, we obtain a 1-twisted vector bundle $\ms{E}^{\wedge}$ on $\ms{X}_{A^{\wedge}}$ of rank $r$. By Artin approximation, we obtain the desired result. \end{proof}

\begin{lemma} \label{20190128-13} Let $S$ be a Noetherian scheme, let $G \to S$ be an affine linearly reductive group scheme. Let $\ms{X} \to \mr{B}G$ be a $\G_{m}$-gerbe such that, for every geometric point $\overline{s} \to S$, the fiber $\ms{X}_{\overline{s}} \to \mr{B}G_{\overline{s}}$ is the trivial $\G_{m}$-gerbe. Then there exists an etale cover $S' \to S$ such that $\ms{X}_{S'} \to \mr{B}G_{S'}$ is trivial, i.e. the image of $[\ms{X}]$ via \[ \H_{\fppf}^{2}(S,\G_{m}) \to \H_{\fppf}^{0}(S,\mb{R}^{2}\pi_{\ast}\G_{m}) \] is trivial. \end{lemma} \begin{proof} We have that $\mr{B}G \to S$ is a good moduli space morphism \cite[12.2]{ALPER-GMSFAS}; thus we have the result by taking $r=1$ in \Cref{20190128-12}. \end{proof}

\begin{theorem}[Mathur] \label{20190912-06} Let $\ms{X}$ be a Noetherian algebraic stack with affine diagonal. Suppose there exists a good moduli space $\ms{X} \to X$. For any $c \in \Br' \ms{X}$, there is an etale surjection $X' \to X$ such that $c|_{\ms{X} \times_{X} X'}$ lies in the image of the Brauer map. \end{theorem} \begin{proof} Let $\mc{G} \to \ms{X}$ be a $\G_{m}$-gerbe. Let $\overline{x} \to X$ be a geometric point, the fiber $\mc{G}_{X_{\overline{x}}}$ is a quotient stack by \cite[Corollary 2.10]{ALPER-HALL-RYDH-LUNA-ETALE-SLICE-THEOREM}. By \cite[3.6]{EdHaKrVi}, the class $[\mc{G}_{X_{\overline{x}}}] \in \H^{2}(\ms{X}_{X_{\overline{x}}},\G_{m})$ lies in the image of the Brauer map of $\ms{X}_{X_{\overline{x}}}$, hence $\mc{G}_{X_{\overline{x}}}$ admits a 1-twisted vector bundle; then we have the desired result by \Cref{20190128-12}. \end{proof}

\begin{corollary} \label{20190912-01} Let $k$ be a base ring, let $G$ be an affine linearly reductive $k$-group scheme acting on a Noetherian $k$-algebra $A$, let $\ms{X} := [(\Spec A)/G]$ be the quotient stack, let $\ms{X} \to X := \Spec A^{G}$ be the good moduli space morphism. If $A^{G}$ is henselian local, then $\Br = \Br'$ for $\ms{X}$. \end{corollary} \begin{proof} Let $\mc{G} \to \ms{X}$ be a $\G_{m,\ms{X}}$-gerbe. By \Cref{20190912-06} there exists an etale surjection $X' \to X$ such that $\mc{G}_{X'}$ admits a 1-twisted vector bundle; since $A^{G}$ is henselian, we may replace $X'$ by a connected component and assume that $X' \to X$ is finite etale; then we have the desired result by \Cref{20160630-wr-01}. \end{proof}

\begin{remark} \label{20190912-08} In particular, \Cref{20190912-01} implies $\Br = \Br'$ for any quotient stack admitting a good moduli space which is the spectrum of a field, e.g. $\ms{X} = [\A_{k}^{n}/\GL_{n}]$ (in characteristic $0$) or $\ms{X} = [\A_{k}^{n}/\G_{m}^{\times n}]$. \end{remark}

\begin{question} Let $\ms{X}$ be an algebraic stack admitting a good moduli space $\pi : \ms{X} \to X$. Is the map \[ \pi^{\ast} : \H_{\fppf}^{2}(X,\G_{m}) \to \H_{\fppf}^{2}(\ms{X},\G_{m}) \] surjective? \Cref{20170920-01} and \Cref{20170920-06} (when $X$ is normal) provide evidence to suggest that the answer is ``yes''. In \cite{MEIER-CBGVCM}, Meier gives a criterion which is sufficient for the vanishing of $\mb{R}^{2}\pi_{\ast}\G_{m}$. (Note that the answer is ``no'' for separated Deligne-Mumford stacks that are not tame, for example $\Br(\A_{\overline{\F_{2}}}^{1}) = 0$ but $\Br(\ms{M}_{1,1,\overline{\F_{2}}}) = \Z/(2)$, see \cite{SHIN-TBGOTMSOECOACFOC2}.) \end{question}

\section{The classifying stack of diagonalizable groups} \label{sec-diagonalizable}

We consider the Brauer groups of classifying stacks by diagonalizable groups. We first describe the unit group and the Picard group of split tori under mild hypotheses on the base scheme.

\begin{lemma}[Sheaf of units on $\G_{m}^{\times n}$ over an integral scheme] \label{20151026-05} Let $X$ be an integral scheme. Then the canonical map \[ \Gamma(X,\G_{m,X}) \oplus \Z^{\oplus n} \to \Gamma(X \times_{\Z} \G_{m,\Z}^{\times n} , \G_{m}) \] is an isomorphism. \end{lemma} \begin{proof} We have the result when $X$ is affine. In general, let $X = \bigcup_{i \in I} X_{i}$ be an affine open cover of $X$. We have a commutative diagram \begin{center}\begin{tikzpicture}[>=stealth] 
\matrix[matrix of math nodes,row sep=2em, column sep=2em, text height=1.5ex, text depth=0.25ex] { 
|[name=11]| 0 & |[name=12]| 0 \\ 
|[name=21]| \Gamma(X,\G_{m}) \oplus \Z^{\oplus n} & |[name=22]| \Gamma(X \times_{\Z} \G_{m,\Z}^{\times n} , \G_{m}) \\
|[name=31]| \prod_{i \in I} \Gamma(X_{i},\G_{m}) \oplus \Z^{\oplus n} & |[name=32]| \prod_{i \in I} \Gamma(X_{i} \times_{\Z} \G_{m,\Z}^{\times n} , \G_{m}) \\
|[name=41]| \prod_{i_{1},i_{2} \in I} \Gamma(X_{i_{1},i_{2}},\G_{m}) \oplus \Z^{\oplus n} & |[name=42]| \prod_{i_{1},i_{2} \in I} \Gamma(X_{i_{1},i_{2}} \times_{\Z} \G_{m,\Z}^{\times n} , \G_{m}) \\
}; 
\draw[->,font=\scriptsize]
(11) edge (21) (21) edge (31) (31) edge (41) (12) edge (22) (22) edge (32) (32) edge (42)
(21) edge node[above=-1pt] {$f_{1}$} (22) (31) edge node[above=-1pt] {$f_{2}$} (32) (41) edge node[above=-1pt] {$f_{3}$} (42); \end{tikzpicture} \end{center} where the columns are equalizer sequences. By the affine case, we have that $f_{2}$ is an isomorphism; hence $f_{1}$ is an injection. Applying this argument to $X_{i_{1},i_{2}}$, we have that $f_{3}$ is an injection. Thus $f_{1}$ is an isomorphism by a diagram chase. \end{proof}

\begin{lemma} \label{20170920-07} Let $\mr{M}$ be a finitely generated torsion-free abelian group, let \[ \mb{T} := \mr{D}(\mr{M})_{\Z} = \Spec \Z[\mr{M}] \] be the associated $\Z$-group scheme. Let $X$ be an integral scheme, let $\mr{B}\mb{T}_{X}$ be the classifying stack, and let $\xi : X \to \mr{B}\mb{T}_{X}$ be the morphism corresponding to the trivial $\mb{T}$-torsor. For $p \ge 0$, let \[ X^{p} := X \times_{\ms{X}} \dotsb \times_{\ms{X}} X \] be the $(p+1)$-fold fiber product of $X$ over $\ms{X}$. The bottom row of the cohomological descent spectral sequence gives a complex \begin{align} \label{20170920-07-eqn-01} \Gamma(X^{0},\G_{m}) \stackrel{\mr{d}^{0}}{\to} \Gamma(X^{1},\G_{m}) \stackrel{\mr{d}^{1}}{\to} \Gamma(X^{2},\G_{m}) \stackrel{\mr{d}^{2}}{\to} \Gamma(X^{3},\G_{m}) \to \dotsb \end{align} of abelian groups. Then \labelcref{20170920-07-eqn-01} is acyclic in degrees $p \ge 1$. \end{lemma} \begin{proof} We have $X^{p} \simeq X \times \mb{T}^{p}$ for all $p$. Since $X$ is an integral scheme, by \Cref{20151026-05} the map \begin{align} \label{20170920-07-eqn-02} \Gamma(X,\G_{m}) \oplus \mr{M}^{\oplus p} \to \Gamma(X \times \mb{T}^{p} , \G_{m}) \end{align} is an isomorphism. With the identification \labelcref{20170920-07-eqn-02}, the differential $\mr{d}^{p}$ is the alternating sum of $p+2$ maps, each of which is the identity on the $\Gamma(X,\G_{m})$ summand; the map $\mr{M}^{\oplus p} \to \mr{M}^{\oplus (p+1)}$ is given by the formula \[ \textstyle \mr{d}^{p}([a_{1},\dotsc,a_{p}]) = [0,a_{1},\dotsc,a_{p}] - (\sum_{i=1}^{p} (-1)^{i} [a_{1},\dotsc,a_{i},a_{i},\dotsc,a_{p}]) + (-1)^{p+1} [a_{1},\dotsc,a_{p},0] \] where ``$[a_{1},\dotsc,a_{i},a_{i},\dotsc,a_{p}]$'' is the vector obtained by replacing ``$a_{i}$'' with ``$a_{i},a_{i}$'' in $[a_{1},\dotsc,a_{p}]$. A computation shows that if $p$ is odd, then the image of $[a_{1},\dotsc,a_{p}]$ under $\mr{d}^{p}$ is given by \[ [0,a_{2},a_{2},a_{4},a_{4},\dotsc,a_{p-1},a_{p-1},0] \] and if $p$ is even, then the image of $\mr{d}^{p}$ is given by \[ [-a_{1},0,a_{2}-a_{3},0,a_{4}-a_{6},0,\dotsc,0,a_{p-2}-a_{p-1},0,a_{p}] \] which gives exactness for $p \ge 1$. \end{proof}

\begin{lemma} \label{20170930-07} \cite[5.10]{BASSMURTH-GGAPGOAGR} Let $A$ be a Noetherian normal ring. Then the pullback \[ \Pic(A) \to \Pic(A[t^{\pm}]) \] is an isomorphism. \end{lemma} \begin{proof} After taking connected components, we may assume that $A$ is a Noetherian normal domain. We have an exact sequence \[ 0 \to \Pic(A) \to \Pic(A[t]) \oplus \Pic(A[t^{-1}]) \to \Pic(A[t^{\pm}]) \to \mr{LPic}(A) \to 0 \] by \cite[Lemma 1.5.1]{WEIBEL-PIACF}, and an isomorphism $\mr{LPic}(A) \simeq \H^{1}_{\fppf}(\Spec A , \Z)$ by \cite[Theorem 5.5]{WEIBEL-PIACF}; we have $\H^{1}_{\fppf}(\Spec A , \Z) = 0$ by \cite[Exp. VIII, Prop. 5.1]{SGA7-I} since $A$ is geometrically unibranch. \end{proof}

\begin{lemma} \label{20170930-06} Let $S$ be a locally Noetherian, integral scheme such that, for every point $s \in S$ of codimension 1, the local ring $\mc{O}_{S,s}$ is regular. Set $\mb{T} := \Spec \Z[t^{\pm}]$ and $\mb{T}_{S} := S \times_{\Spec \Z} \mb{T}$, and let $\pi : \mb{T}_{S} \to S$ be the projection. Then the pullback map \[ \pi^{\ast} : \Pic(S) \to \Pic(\mb{T}_{S}) \] is an isomorphism. \end{lemma} \begin{proof}[Proof 1] \footnote{Following comments by user ``Heer'' in \url{https://mathoverflow.net/q/84414}.} We check the conditions of \cite[$\text{IV}_{4}$, (21.4.9)]{EGA}. The projection $\pi$ is faithfully flat and has a section, hence $\pi^{\ast}$ is injective; the map $\pi$ is both quasi-compact and open. Given a codimension 1 point $s \in S$, set $A := \mc{O}_{S,s}$; since $A$ is seminormal, the pullback $\Pic(A) \to \Pic(\A_{A}^{1})$ is an isomorphism by Traverso's theorem \cite[Theorem 3.6]{TRAVERSO-SAPG}; since $A$ is regular, for any open subscheme $U \subseteq \A_{A}^{1}$ we have an isomorphism $\Pic(U) \simeq \Cl(U)$; the restriction map $\Cl(\A_{A}^{1}) \to \Cl(U)$ is surjective; we take $U := \Spec A[t^{\pm}]$. \end{proof} \begin{proof}[Proof 2, if $S$ is normal and quasi-compact] After taking connected components, we may assume that $S$ is a Noetherian normal integral scheme. Since the projection $\pi$ has a section, it is clear that $\pi^{\ast}$ is injective. For any quasi-compact scheme $Y$, let $n(Y)$ be the minimal size of an affine open covering of $Y$. \par We proceed by induction on $n(S)$. The case $n(S) = 1$ (in other words $S$ is affine) is \Cref{20170930-07}. \par In general, suppose $S = S_{1} \cup S_{2}$ where $S_{1},S_{2}$ are open subschemes of $S$ such that $n(S_{i}) < n(S)$. Let $\pi_{i} : \mb{T}_{S_{i}} \to S_{i}$ be the projections. Suppose $\mc{L}$ is an invertible sheaf on $\mb{T}_{S}$; by the induction hypothesis, there exist invertible $\mc{O}_{S_{i}}$-modules $\mc{M}_{i}$ and isomorphisms \[ \varphi_{i} : \mc{L}|_{\mb{T}_{S_{i}}} \to \pi_{i}^{\ast}\mc{M}_{i} \] of $\mc{O}_{\mb{T}_{S_{i}}}$-modules. Set $S_{12} := S_{1} \cap S_{2}$ and $\pi_{12} : \mb{T}_{S_{12}} \to S_{12}$ the projection; since $\Pic(S_{12}) \to \Pic(\mb{T}_{S_{12}})$ is injective, there is an isomorphism \[ \alpha : \mc{M}_{1}|_{S_{12}} \to \mc{M}_{2}|_{S_{12}} \] of $\mc{O}_{S_{12}}$-modules; moreover, since the inclusion \[ \Gamma(S_{12},\G_{m}) \times t^{\Z} \to \Gamma(\mb{T}_{S_{12}},\G_{m}) \] is an isomorphism (by \Cref{20151026-05}), we may multiply $\alpha$ by a unit in $\G_{m}(S_{12})$ and multiply $\mc{M}_{1}$ by a character $t^{n}$ so that $\pi_{12}^{\ast}\alpha$ corresponds to $(\varphi_{2}|_{\mb{T}_{S_{12}}}) \circ (\varphi_{1}|_{\mb{T}_{S_{12}}})^{-1}$. Thus the invertible $\mc{O}_{S}$-module obtained by gluing $\mc{M}_{1},\mc{M}_{2}$ along $\alpha$ gives the desired element of $\Pic(S)$ whose image in $\Pic(\mb{T}_{S})$ is $\mc{L}$. \end{proof}

\begin{lemma} \label{20170920-01} Let $X$ be a scheme. Let $\mathrm{M}$ be a finitely generated torsion-free abelian group, let $\mb{T} := \mathrm{D}(\mathrm{M})$ be the associated diagonalizable group scheme, and let $\pi : \mathrm{B}\mb{T}_{X} \to X$ be the classifying stack. Then there is an exact sequence \[ 0 \to \H_{\fppf}^{2}(X,\G_{m}) \stackrel{\pi^{\ast}}{\to} \H_{\fppf}^{2}(\mr{B}\mb{T}_{X},\G_{m}) \to \H_{\fppf}^{1}(X,\underline{\mr{M}}) \] of groups. \end{lemma} \begin{proof} Suppose first that $X$ is Noetherian normal. Let $\xi : X \to \mr{B}\mb{T}_{X}$ be the morphism corresponding to the trivial $\mb{T}$-torsor. The descent spectral sequence \Cref{20151211-04} associated to the covering $\xi$ is of the form \begin{align} \label{20170920-01-eqn-01} \mathrm{E}_{1}^{p,q} = \H_{\fppf}^{q}(\mb{T}_{X}^{p} , \G_{m}) \implies \H_{\fppf}^{p+q}(\mr{B}\mb{T}_{X},\G_{m}) \end{align} with differentials $\mathrm{d}_{1}^{p,q} : \mathrm{E}_{1}^{p,q} \to \mathrm{E}_{1}^{p+1,q}$. \begin{center}\begin{tikzpicture}[>=stealth] 
\matrix[matrix of math nodes,row sep=1em, column sep=1.5em, text height=1.5ex, text depth=0.25ex] { 
|[name=01]| \vdots & |[name=02]| \vdots & |[name=03]| \vdots & |[name=04]| \vdots & |[name=05]| \\
|[name=11]| \H_{\fppf}^{3}(\mb{T}_{X}^{0},\G_{m}) & |[name=12]| \H_{\fppf}^{3}(\mb{T}_{X}^{1},\G_{m}) & |[name=13]| \H_{\fppf}^{3}(\mb{T}_{X}^{2},\G_{m}) & |[name=14]| \H_{\fppf}^{3}(\mb{T}_{X}^{3},\G_{m}) & |[name=15]| \dotsb \\ 
|[name=21]| \H_{\fppf}^{2}(\mb{T}_{X}^{0},\G_{m}) & |[name=22]| \H_{\fppf}^{2}(\mb{T}_{X}^{1},\G_{m}) & |[name=23]| \H_{\fppf}^{2}(\mb{T}_{X}^{2},\G_{m}) & |[name=24]| \H_{\fppf}^{2}(\mb{T}_{X}^{3},\G_{m}) & |[name=25]| \dotsb \\ 
|[name=31]| \H_{\fppf}^{1}(\mb{T}_{X}^{0},\G_{m}) & |[name=32]| \H_{\fppf}^{1}(\mb{T}_{X}^{1},\G_{m}) & |[name=33]| \H_{\fppf}^{1}(\mb{T}_{X}^{2},\G_{m}) & |[name=34]| \H_{\fppf}^{1}(\mb{T}_{X}^{3},\G_{m}) & |[name=35]| \dotsb \\ 
|[name=41]| \H_{\fppf}^{0}(\mb{T}_{X}^{0},\G_{m}) & |[name=42]| \H_{\fppf}^{0}(\mb{T}_{X}^{1},\G_{m}) & |[name=43]| \H_{\fppf}^{0}(\mb{T}_{X}^{2},\G_{m}) & |[name=44]| \H_{\fppf}^{0}(\mb{T}_{X}^{3},\G_{m}) & |[name=45]| \dotsb \\ 
}; 
\draw[->,font=\scriptsize]
(11) edge (12) (12) edge (13) (13) edge (14) (14) edge (15)
(21) edge (22) (22) edge (23) (23) edge (24) (24) edge (25)
(31) edge (32) (32) edge (33) (33) edge (34) (34) edge (35)
(41) edge (42) (42) edge (43) (43) edge (44) (44) edge (45); \end{tikzpicture} \end{center} Note that each differential $\mathrm{d}_{1}^{0,q} : \mathrm{E}_{1}^{0,q} \to \mathrm{E}_{1}^{1,q}$ is the $0$ map since the two projection maps $\mb{T}_{X}^{1} \rightrightarrows X$ are equal (since $\mb{T}$ acts trivially on $X$). By \Cref{20151026-05}, the map \begin{align} \label{20170920-01-eqn-02} \Gamma(X,\G_{m}) \oplus \mathrm{M}^{\oplus p} \to \Gamma(\mb{T}_{X}^{p} , \G_{m}) \end{align} is an isomorphism. Note that there are $p+2$ projection maps $\mb{T}^{p+1} \to \mb{T}^{p}$. Since $X$ is Noetherian normal, by \Cref{20170930-06} we have that $\mathrm{d}_{1}^{p,1} : \mathrm{E}_{1}^{p,1} \to \mathrm{E}_{1}^{p+1,1}$ is $0$ if $p$ is even and an isomorphism if $p$ is odd; thus $\mathrm{E}_{2}^{p,1} = 0$ for $p \ge 1$. \par Via the identification \labelcref{20170920-01-eqn-02}, we obtain a complex $\mathrm{E}_{1}^{\bullet,0}$ which is exact in degrees $p \ge 1$ by \Cref{20170920-07}. The above considerations show that the desired map $\pi^{\ast}$ is an isomorphism. \par If $X$ is not normal, we have $\mb{R}^{2}\pi_{\ast}\G_{m} = 0$ by \Cref{20190128-13} and the above normal case. We have $\mb{R}^{1}\pi_{\ast}\G_{m} \simeq \mb{T}^{\vee} \simeq \underline{\mr{M}}$ by \cite[2.10]{SHIN-TCBGOATGMG}, so the Leray spectral sequence for $\pi$ gives the desired exact sequence, where $\pi^{\ast}$ is injective since $\xi^{\ast}\pi^{\ast} = \id$. \end{proof}

\begin{theorem} \label{20170920-03} Let $\mr{M}$ be a finitely generated abelian group, and let $D \to \Spec \Z$ be the Cartier dual of $\mr{M}$. Then the morphism $\mr{B}D \to \Spec \Z$ satisfies $(\mr{SBMI})$. \end{theorem} \begin{proof} Suppose $\Br = \Br'$ for $X$, and let $c \in \Br'(\mr{B}D_{X})$ be a Brauer class. There is an exact sequence $0 \to \mr{N} \to \mr{M} \to \Z^{\oplus r} \to 0$ for some finite abelian group $\mr{N}$; taking the dual gives an exact sequence $1 \to \mb{T} \to D \to \mr{D}(\mr{N}) \to 1$ of group schemes. By \cite[4.4]{SHIN-TCBGOATGMG}, the group $\H_{\fppf}^{1}(X,\underline{\mr{M}})$ is torsion-free, so By \Cref{20170920-01}, the pullback map $\Br' X \to \Br' \mr{B}\mb{T}_{X}$ is an isomorphism, so $c|_{\mr{B}\mb{T}_{X}}$ is contained in $\Br(\mr{B}\mb{T}_{X})$ by \Cref{20160630-wr-02}; hence $c$ is contained in $\Br(\mr{B}D_{X})$ by \Cref{20160630-wr-03}. \end{proof}

\begin{remark} \label{20170920-15} One difficulty in working with torsion-free abelian groups $\mr{M}$ of higher rank is that $\Pic(A[\mr{M}])$ can be large if $A$ is not seminormal (see Weibel's description in \cite{WEIBEL-PIACF}). In \cite{SHIN-TCBGOATGMG} we prove $\mb{R}^{2}\pi_{\ast}\G_{m} = 0$ (in case $\mr{M} = \Z$) in a different way, by computing the translation-invariant subgroup of $\G_{m}$ instead of using \Cref{20190128-13}. \end{remark}

\section{The classifying stack of $\mathrm{GL}_{n}$} \label{sec-gln}

In this section we compute the cohomological Brauer group of $\GL_{n}$ over normal schemes.

\begin{setup} \label{20171109-21} Let $A$ be a ring, let \[ \mr{X}_{\bullet} := \{\mr{X}_{i,j}\}_{i,j=1,\dotsc,n} \] be a collection of $n^{2}$ variables, let $A[\mr{X}_{\bullet}]$ be the polynomial ring, let \[ \det \in A[\mr{X}_{\bullet}] \] be the determinant of the $n \times n$ matrix whose $(i,j)$th entry is $\mr{X}_{i,j}$. The localization $A[\mr{X}_{\bullet},\frac{1}{\det}]$ may be identified with the coordinate ring of $\GL_{n,A}$. \end{setup}

\begin{lemma} \label{20171109-05} The map \[ A \to A[\mr{X}_{\bullet}]/(\det) \] is faithfully flat. \end{lemma} \begin{proof}[Proof 1] We may assume $A = \Z$ since faithfully flat morphisms are preserved by base change. For flatness, it suffices to show that $\Z[\mr{X}_{\bullet}]/(\det)$ is torsion-free. Suppose $\ell \in \Z$ and $a \in \Z[\mr{X}_{\bullet}]$ such that $\ell a \in (\det)$; since $\Z[\mr{X}_{\bullet}]/(\det)$ is an integral domain \cite[(2.10) Theorem]{BRUNSVETTER-DETERMINANTALRINGS}, either $\ell \in (\det)$ or $a \in (\det)$, but it is not possible that $\ell \in \det$ since $\ell$ has degree $0$ whereas $\det$ has degree $n$. Since $\Z \to \Z[\mr{X}_{\bullet}]/(\det)$ has a retraction, it is faithfully flat. \end{proof} \begin{proof}[Proof 2] We can make a change of coordinates $\mr{X}_{i,i} \mapsto \mr{X}_{i,i}+\mr{X}_{1,1}$ for $i \ge 2$. Let $f$ be the polynomial that $\det$ gets sent to; then $f$ is monic of degree $n$ in the variable $\mr{X}_{1,1}$, hence $A[\mr{X}_{\bullet}]/(\det)$ is finite locally free over $A[\mr{X}_{\bullet} \setminus \{\mr{X}_{1,1}\}]$, which is smooth over $A$. \end{proof}

\begin{lemma} \label{20171109-06} The element $\det$ is a nonzerodivisor of $A[\mr{X}_{\bullet}]$. \end{lemma} \begin{proof} Since $\Z[\mr{X}_{\bullet}]/(\det)$ is an integral domain, we have that $\det$ is irreducible element of $\Z[\mr{X}_{\bullet}]$; hence it is a nonzerodivisor on $\Z[\mr{X}_{\bullet}]$; hence the sequence \begin{align} \label{20171109-06-eqn-01} 0 \to \Z[\mr{X}_{\bullet}] \to \Z[\mr{X}_{\bullet}] \to \Z[\mr{X}_{\bullet}]/(\det) \to 0 \end{align} is exact; here $\Z[\mr{X}_{\bullet}]/(\det)$ is flat over $\Z$ by \Cref{20171109-05}; tensoring \labelcref{20171109-06-eqn-01} with $-\otimes_{\Z} A$ gives \[ 0 \to A[\mr{X}_{\bullet}] \stackrel{\ast}{\to} A[\mr{X}_{\bullet}] \to A[\mr{X}_{\bullet}]/(\det) \to 0 \] where the map $\ast$ is injective by e.g. \cite[00HL]{SP}. \end{proof}

\begin{lemma} \label{20171109-07} The map \begin{align} \Phi_{A} : \label{20171109-21-eqn-01} \textstyle A^{\times} \oplus \Gamma(\Spec A , \underline{\Z}) \to (A[\mr{X}_{\bullet},\frac{1}{\det}])^{\times} \end{align} sending $(a,n) \mapsto a \det^{n}$ is injective. \end{lemma} \begin{proof}[Proof] This follows from \Cref{20171109-06}. \end{proof}

\begin{lemma} \label{20171109-08} If $A$ is an integral domain, the map \labelcref{20171109-21-eqn-01} is an isomorphism. \end{lemma} \begin{proof}[Proof] Suppose $\frac{a_{1}}{\det^{f_{1}}}$ is a unit of $A[\mr{X}_{\bullet},\frac{1}{\det}]$, with inverse $\frac{a_{2}}{\det^{f_{2}}}$. Then $a_{1}a_{2} = \det^{f_{1}+f_{2}}$ since $\det$ is a nonzerodivisor \Cref{20171109-06}. Since $A$ is an integral domain, we may assume that $a_{1},a_{2}$ are homogeneous. We have that $\det$ is a prime element of $A[\mr{X}_{\bullet}]$ by \cite[(2.10) Theorem]{BRUNSVETTER-DETERMINANTALRINGS}. \end{proof}

\begin{lemma}[{Units of $\GL_{n}$}] \label{20171109-02} \footnote{Broughton \cite{BROUGHTON-ANOCOAG1983} shows that the units of the coordinate ring of an algebraic group over any algebraically closed field are given by characters.} The map $\Phi_{A}$ \labelcref{20171109-21-eqn-01} is an isomorphism if and only if $A$ is reduced. \end{lemma} \begin{proof} If $A$ is an integral domain, we have that $\Phi_{A}$ is an isomorphism by \Cref{20171109-08}. More generally, if $A$ is the finite product of integral domains, then $\Phi_{A}$ is an isomorphism. \par Suppose that $A$ is reduced. By limit arguments, we may assume that $A$ is (reduced and) a finite type $\Z$-algebra. We may assume that $\Spec A$ is connected. Let $\mf{p}_{1},\dotsc,\mf{p}_{r}$ be the minimal primes of $A$. Then the total ring of fractions of $A$ is \[ \mr{Q}(A) = k(\mf{p}_{1}) \oplus \dotsb \oplus k(\mf{p}_{r}) \] by e.g. \cite[02LX]{SP}. Let \[ \frac{\beta_{1}}{\det^{f_{1}}} , \frac{\beta_{2}}{\det^{f_{2}}} \] be two elements of $A[\mr{X}_{\bullet},\frac{1}{\det}]$ with $\beta_{i} \in A[\mr{X}_{\bullet}]$ and $f_{1},f_{2} \in \Z_{\ge 0}$ such that \[ \frac{\beta_{1}\beta_{2}}{\det^{f_{1}+f_{2}}} = 1 \] in $A[\mr{X}_{\bullet},\frac{1}{\det}]$. Then $\beta_{1}\beta_{2}\det^{f_{3}} = \det^{f_{1}+f_{2}+f_{3}}$ in $A[\mr{X}_{\bullet}]$ for some $f_{3}$, but $\det$ is a nonzerodivisor in $A[\mr{X}_{\bullet}]$ (by \Cref{20171109-06}) so \begin{align} \label{20171109-02-eqn-01} \beta_{1}\beta_{2} = {\det}^{f_{1}+f_{2}} \end{align} in $A[\mr{X}_{\bullet}]$. Plugging in $\mr{X}_{\bullet} = T \cdot \id_{n}$ for a variable $T$ into \labelcref{20171109-02-eqn-01} gives \[ \beta_{1}|_{T \cdot \id_{n}} \cdot \beta_{2}|_{T \cdot \id_{n}} = T^{n(f_{1}+f_{2})} \] so $\beta_{1}|_{T \cdot \id_{n}},\beta_{2}|_{T \cdot \id_{n}}$ are units of $A[T^{\pm}]$; thus (since $A$ is connected and reduced) we have by \cite[Corollary 6]{NEHER-IANEITGAOAUPG} that $\beta_{1}|_{T \cdot \id_{n}},\beta_{2}|_{T \cdot \id_{n}}$ are homogeneous. \par The image of $\beta_{i}$ in $(\mr{Q}(A)[\mr{X}_{\bullet},\frac{1}{\det}])^{\times}$ is contained in the image of $\Phi_{\mr{Q}(A)}$ so by limit arguments there exists a nonzerodivisor $s_{i} \in A$ such that the image of $\beta_{i}$ in $(A[\frac{1}{s_{i}}][\mr{X}_{\bullet},\frac{1}{\det}])^{\times}$ is contained in the image of $\Phi_{A[\frac{1}{s_{i}}]}$; in other words, there exist $a_{i,1},\dotsc,a_{i,m_{i}} \in A[\frac{1}{s_{i}}]$ (say $a_{i,1} \ne 0$) and integers $0 \le e_{i,1} < \dotsb < e_{i,m_{i}}$ such that \[ \textstyle \beta_{i} = \sum_{\ell=1}^{m_{i}} a_{i,\ell}{\det}^{e_{i,\ell}} \] in $A[\frac{1}{s_{i}}][\mr{X}_{\bullet}]$; here $\beta_{i} \in A[\mr{X}_{\bullet}]$ implies $a_{i,\ell} \in A$ for all $\ell$. Since $\beta_{i}|_{T \cdot \id_{n}} = \sum_{\ell=1}^{m_{i}} a_{\ell}T^{ne_{i,\ell}}$ is homogeneous in $A[\frac{1}{s_{i}}][\mr{X}_{\bullet}]$, all but one $a_{i,\ell}$ is nonzero, in other words $\beta_{i} = a_{i,1}\det^{e_{i,1}}$. This means \[ a_{1,1}a_{2,1}{\det}^{e_{1,1}+e_{2,1}} = {\det}^{f_{1}+f_{2}} \] in $A[\frac{1}{s_{1}s_{2}}][\mr{X}_{\bullet}]$; thus $a_{1,1}a_{2,1} = 1$ in $A$, so $a_{1,1},a_{2,1}$ are units of $A$. \par (Thanks to Justin Chen for pointing out the following.) If $a \in A$ is nonzero and satisfies $a^{2} = 0$, then \[ (\det+a)(\det-a) = {\det}^{2} \] so $\det+a$ is a unit of $A[\mr{X}_{\bullet},\frac{1}{\det}]$ which is not in the image of \labelcref{20171109-21-eqn-01}. \end{proof}

\begin{lemma} \label{20170930-48} Let $S$ be a locally Noetherian, integral scheme such that, for every point $s \in S$ of codimension 1, the local ring $\mc{O}_{S,s}$ is regular. For any positive integer $p$, the pullback \begin{align} \label{20170930-48-eqn-01} \Pic(S) \to \Pic(S \times_{\Spec \Z} (\GL_{n,\Z})^{\times p}) \end{align} is an isomorphism. \end{lemma} \begin{proof} We check the conditions of \cite[$\text{IV}_{4}$, (21.4.9)]{EGA}. Let $\pi : S \times_{\Spec \Z} (\GL_{n,\Z})^{\times p} \to S$ be the projection; it is faithfully flat and has a section, hence \labelcref{20170930-48-eqn-01} is injective; the map $\pi$ is both quasi-compact and open. Given a codimension 1 point $s \in S$, set $A := \mc{O}_{S,s}$; since $A$ is seminormal, the pullback $\Pic(A) \to \Pic(\A_{A}^{pn^{2}})$ is an isomorphism; since $A$ is regular, for any open subscheme $U \subseteq \A_{A}^{pn^{2}}$ we have an isomorphism $\Pic(U) \simeq \Cl(U)$; the restriction map $\Cl(\A_{A}^{pn^{2}}) \to \Cl(U)$ is surjective; we take $U := \Spec A \times_{\Spec \Z} (\GL_{n,\Z})^{\times p}$. \end{proof}

\begin{proposition}[Brauer group of classifying stack $\mathrm{BGL}_{n}$] \label{20170920-06} Let $S$ be a locally Noetherian, integral scheme such that, for every point $s \in S$ of codimension 1, the local ring $\mc{O}_{S,s}$ is regular. Let $\xi : S \to \mr{B}\GL_{n,S}$ be the morphism corresponding to the trivial $\GL_{n}$-torsor. Then the pullback map \[ \xi^{\ast} : \H_{\fppf}^{2}(\mr{B}\GL_{n,S},\G_{m}) \to \H_{\fppf}^{2}(S , \G_{m}) \] is an isomorphism. \end{proposition} \begin{proof} We set $G := \GL_{n,S}$ for convenience. The cohomological descent spectral sequence associated to the covering $\xi : S \to \mr{B}G$ gives a spectral sequence \begin{align} \label{20170920-06-eqn-01} \mr{E}_{1}^{p,q} = \H_{\fppf}^{q}(G^{p} , \G_{m}) \implies \H_{\fppf}^{p+q}(\mr{B}G,\G_{m}) \end{align} with differentials $\mr{d}_{1}^{p,q} : \mr{E}_{1}^{p,q} \to \mr{E}_{1}^{p+1,q}$. \begin{center}\begin{tikzpicture}[>=stealth] 
\matrix[matrix of math nodes,row sep=1em, column sep=1.5em, text height=1.5ex, text depth=0.25ex] { 
|[name=01]| \vdots & |[name=02]| \vdots & |[name=03]| \vdots & |[name=04]| \vdots & |[name=05]| \\
|[name=11]| \H_{\fppf}^{3}(G^{0},\G_{m}) & |[name=12]| \H_{\fppf}^{3}(G^{1},\G_{m}) & |[name=13]| \H_{\fppf}^{3}(G^{2},\G_{m}) & |[name=14]| \H_{\fppf}^{3}(G^{3},\G_{m}) & |[name=15]| \dotsb \\ 
|[name=21]| \H_{\fppf}^{2}(G^{0},\G_{m}) & |[name=22]| \H_{\fppf}^{2}(G^{1},\G_{m}) & |[name=23]| \H_{\fppf}^{2}(G^{2},\G_{m}) & |[name=24]| \H_{\fppf}^{2}(G^{3},\G_{m}) & |[name=25]| \dotsb \\ 
|[name=31]| \H_{\fppf}^{1}(G^{0},\G_{m}) & |[name=32]| \H_{\fppf}^{1}(G^{1},\G_{m}) & |[name=33]| \H_{\fppf}^{1}(G^{2},\G_{m}) & |[name=34]| \H_{\fppf}^{1}(G^{3},\G_{m}) & |[name=35]| \dotsb \\ 
|[name=41]| \H_{\fppf}^{0}(G^{0},\G_{m}) & |[name=42]| \H_{\fppf}^{0}(G^{1},\G_{m}) & |[name=43]| \H_{\fppf}^{0}(G^{2},\G_{m}) & |[name=44]| \H_{\fppf}^{0}(G^{3},\G_{m}) & |[name=45]| \dotsb \\ 
}; 
\draw[->,font=\scriptsize]
(11) edge (12) (12) edge (13) (13) edge (14) (14) edge (15)
(21) edge (22) (22) edge (23) (23) edge (24) (24) edge (25)
(31) edge (32) (32) edge (33) (33) edge (34) (34) edge (35)
(41) edge (42) (42) edge (43) (43) edge (44) (44) edge (45); \end{tikzpicture} \end{center} Note that each differential $\mr{d}_{1}^{0,q} : \mr{E}_{1}^{0,q} \to \mr{E}_{1}^{1,q}$ is the $0$ map since the two projection maps $G \rightrightarrows S$ are equal (since $G$ acts trivially on $S$). By \Cref{20171109-08}, there is an isomorphism \begin{align} \label{20170920-06-eqn-02} A^{\times} \oplus \Z^{\oplus p} \to \Gamma(G^{p} , \G_{m}) \end{align} sending the $i$th generator to the determinant of the $i$th component of $G^{p}$, and furthermore the complex $\mr{E}_{1}^{\bullet,0}$ becomes identified with the corresponding complex for $\mr{B}\G_{m}$ via the map $\mr{B}G \to \mr{B}\G_{m}$ defined by the determinant, hence is acyclic in degrees $p \ge 1$ by the argument in \Cref{20170920-07}. Moreover we have $\mr{E}_{1}^{p,1} = \H_{\fppf}^{1}(G^{p},\G_{m}) = \Pic(G^{p})$ and the pullback $\Pic(S) \to \Pic(G^{p})$ is an isomorphism for all $p \ge 0$ by \Cref{20170930-48}. The above considerations show that the canonical pullback map \[ \H_{\fppf}^{2}(\mr{B}G,\G_{m}) \to \H_{\fppf}^{2}(S,\G_{m}) \] is an isomorphism. \end{proof}

\begin{corollary} \label{20170920-18} Assume the setup of \Cref{20170920-06}. Then $\Br = \Br'$ for $S$ if and only if $\Br = \Br'$ for $\mr{B}\GL_{n,S}$. \end{corollary} \begin{proof} The argument of \Cref{20170920-03} applies, using \Cref{20170920-06}. \end{proof}

\begin{remark} \label{20171109-29} Note that \Cref{20190128-13} does not apply for $\GL_{n}$ since it is not linearly reductive over $\Z$ \cite[12.4]{ALPER-GMSFAS}. \end{remark}

\bibliography{../allbib}
\bibliographystyle{alpha}

\end{document}